\documentclass{amsart}
\usepackage{amsmath, amsfonts, amssymb}

\def\a{\alpha}

\def\g{\gamma}
\def\G{\Gamma}

\newcommand{\mB}{\mathcal{B}}
\newcommand{\mC}{\mathcal{C}}

\newcommand{\mH}{\mathcal{H}}

\newcommand{\mM}{\mathcal{M}}

\newcommand{\fp}{\mathfrak{p}}

\newcommand{\bfC}{\mathbf{C}}

\newcommand{\bfG}{\mathbf{G}}

\newcommand{\bfQ}{\mathbf{Q}}
\newcommand{\bfR}{\mathbf{R}}

\newcommand{\bfT}{\mathbf{T}}

\newcommand{\bfZ}{\mathbf{Z}}

\newcommand{\bff}{\mathbf{f}}

\newcommand{\bfi}{\mathbf{i}}

\newcommand{\Oo}{\mathcal{O}}

\newcommand{\OF}{\mathcal{O}_F}

\newcommand{\AF}{\mathbf{A}_F}

\newcommand{\AQ}{\mathbf{A}}

\newcommand{\AFf}{\mathbf{A}_{F,\textup{f}}}

\newcommand{\AQf}{\mathbf{A}_{\textup{f}}}

\newcommand{\tuf}{\textup{f}}

\newcommand{\tuM}{\textup{M}}

\newcommand{\ov}{\overline}

\newcommand{\be}{\begin{equation}}
\newcommand{\ee}{\end{equation}}
\newcommand{\bes}{\begin{equation*}}
\newcommand{\ees}{\end{equation*}}

\newcommand{\bs}{\begin{split}}
\newcommand{\es}{\end{split}}
\newcommand{\bss}{\begin{split*}}
\newcommand{\ess}{\end{split*}}

\newcommand{\bmat}{\left[ \begin{matrix}}
\newcommand{\emat}{\end{matrix} \right]}
\newcommand{\bsmat}{\left[ \begin{smallmatrix}}
\newcommand{\esmat}{\end{smallmatrix} \right]}

\newcommand{\bml}{\begin{multline}}
\newcommand{\eml}{\end{multline}}
\newcommand{\bmls}{\begin{multline*}}
\newcommand{\emls}{\end{multline*}}

\DeclareMathOperator{\Cl}{Cl}

\DeclareMathOperator{\Desc}{Desc}
\DeclareMathOperator{\diag}{diag}

\DeclareMathOperator{\Gal}{Gal}
\DeclareMathOperator{\GL}{GL}

\DeclareMathOperator{\GU}{GU}

\DeclareMathOperator{\Res}{Res}

\DeclareMathOperator{\SL}{SL}

\DeclareMathOperator{\U}{U}
\DeclareMathOperator{\val}{val}

\def\AQf{\mathbf{A}_{\textup{f}}}

\newcommand{\hs}{\hspace{2pt}}
\newcommand{\hf}{\hspace{5pt}}

\newcommand{\iy}{\infty}

\newcommand{\tr}{\textup{tr}\hspace{2pt}}

\usepackage[all]{xy}
\usepackage{amsthm}
\SelectTips{cm}{10}\UseTips
\theoremstyle{plain}
\newtheorem{thm}{Theorem}
\newtheorem{prop}[thm]{Proposition}
\newtheorem{cor}[thm]{Corollary}
\newtheorem{lemma}[thm]{Lemma}

\theoremstyle{definition}
\newtheorem{definition}[thm]{Definition}

\newtheorem{rem}[thm]{Remark}

\numberwithin{thm}{section}
\numberwithin{equation}{section}

\begin{document}

\title[Adelic Maass spaces]{Adelic Maass spaces on $\U(2,2)$}
\author{Krzysztof Klosin}
\date{June 8, 2007}

\maketitle

\begin{abstract} Generalizing the results of \cite{Kojima}, 
\cite{Gritsenko90} and \cite{Krieg91}, we define an adelic version of the 
Maass 
space for hermitian modular forms of weight $k$ regarded as functions on 
adelic 
points of the quasi-split 
unitary group $\U(2,2)$ associated with an imaginary quadratic extension 
$F/\bfQ$ of discriminant $D_F$. When the class number $h_F$ of $F$ is odd, 
we show that the 
Maass 
space is invariant under the action of the local Hecke algebras of 
$\U(2,2)(\bfQ_p)$ for 
all $p \nmid D_F$. As a consequence we obtain a Hecke-equivariant 
injective map 
from the Maass space to the $h_F$-fold direct product of the space of 
elliptic 
modular forms $M_{k-1}(\G_0(D_F))$.
\end{abstract}

\section{Introduction}

In 1977 Saito and Kurokawa \cite{Kurokawa78} conjectured that cusp forms 
in $S_{2k-2}(\SL_2(\bfZ))$ can be lifted to Siegel modular forms of 
weight 
$k$ 
and that this lifting is Hecke-equivariant in an appropriate sense. This 
conjecture was proved in a series of papers by Maass \cite{Maass79}, 
Andrianov \cite{Andrianov79}, and Zagier \cite{Zagier81}, and later 
reformulated and proved using the language of representation theory by 
Piatetski-Shapiro \cite{Piatetski-Shapiro83}.
In the 1980s Kojima \cite{Kojima} and Gritsenko \cite{Gritsenko90} proved 
the
existence of a similar lifting from the space of modular forms of level 4
and non-trivial character to the space of hermitian modular forms.
The group $\U(2,2)$ for which the hermitian modular forms were defined was 
associated with
the extension $\bfQ(i)/\bfQ$. Following \cite{Kojima} we will refer to the 
image
of this lifting as the \textit{Maass space} and to the lifting itself as 
the \textit{Maass lifting}. In 1991 Krieg
\cite{Krieg91} extended the results of \cite{Kojima} and 
\cite{Gritsenko90} to imaginary quadratic fields and
showed Hecke-equivariance of the Maass lifting for a certain family of 
Hecke
operators living in the local Hecke algebra at $p$, for $p$ any
inert prime. As \cite{Kojima}, \cite{Gritsenko90} and \cite{Krieg91}
all treated hermitian modular forms as classical objects (i.e., as
functions defined on a higher-dimensional analogue of the complex upper
half-plane), their methods did not yield Hecke-equivariance of the lift 
under
the action of local Hecke algebras at split primes, except when the class
number of $F$ was one (\cite{Gritsenko90}).

In this article we define an adelic version of the Maass space by 
imposing a condition on the Fourier coefficients of hermitian modular 
forms regarded as functions on $\U(2,2)(\AQ)$ (Definition \ref{Maass 
form}). 
In fact our definition reduces to that of Krieg's when the
class number of $F$ is one. We also show that the adelic Maass space, 
which we 
denote by $\mM_k$, is invariant
under the action of local Hecke algebras at all primes $p$ not ramified 
in $F/\bfQ$ (Theorem \ref{thmmain}).

To prove Hecke-equivariance of the Maass space we assume that the class 
number $h_F$ of $F$ is odd, and
in that
case we show that $\mM_k$ is isomorphic to $h_F$
copies of
Krieg's Maass space (Proposition \ref{oddclass}), hence we obtain a 
lifting from
$M_{k-1}(D_F,
\chi_F)^{h_F}$ to $\mM_k$ together with a homomorphism from the Hecke
algebra of $\mM_k$ into the space of endomorphisms of $M_{k-1}(D_F,
\chi_F)^{h_F}$ generated by the classical Hecke operators $T_p$ (for $p$
split), $T^2_{p}$ (for $p$ inert) and the group of permutations of the
factors in $M_{k-1}(D_F,
\chi_F)^{h_F}$ (Theorems \ref{heckedesc} and \ref{heckedesc2}).

In \cite{Klosin06preprint} the author used Hecke-equivariance of the Maass 
lifting 
for the 
extension $\bfQ(i)/\bfQ$ proved in \cite{Gritsenko90} to construct 
congruences between Maass forms and non-Maass forms and as a consequence 
give evidence for the Bloch-Kato conjecture (Theorem 7.11 and Corollary 
9.9 in 
\cite{Klosin06preprint}). We hope to be able to use Theorem \ref{thmmain} 
to 
extend the results of \cite{Klosin06preprint} to all imaginary quadratic 
fields of 
odd class number.

We believe that our assumption on the class number of $F$ is unnecessary
and we make it here only to make the proofs simpler. We work with the
Hecke algebra for the algebraic group $\U(2,2)$ instead of $\GU(2,2)$
which is the case in \cite{Gritsenko90} and \cite{Krieg91}, but our
results easily imply analogous ones for $\GU(2,2)$. Finally we want to
point out that
the Maass lifting we obtain agrees with a lifting that has recently   
been defined by Ikeda
\cite{Ikedapreprint2005} using a different approach. Our result provides
an explicit description of the image of that lifting as well as explicit
formulas for the descent of Hecke algebras.

\section{Definitions}

Let $F$ be an imaginary quadratic field with ring of integers $\OF$.
We denote by $\Cl_F$ the class group of $F$ and set $h_F:= \# \Cl_F$. For
any affine group scheme $X$ over $\OF$ and any $\bfZ$-algebra $A$, we
denote by $x \mapsto \ov{x}$
the automorphism of $(\Res_{\OF/\bfZ}X)(A)$ induced by the non-trivial
automorphism of
$F/\bfQ$. Note that $(\Res_{\OF/\bfZ}X)(A)$ can be identified with a 
subgroup of $\GL_n(A \otimes \OF)$ for some $n$. In what follows we 
always specify such an identification. Then for $x \in 
(\Res_{\OF/\bfZ}X)(A)$ we write $x^t$ for the transpose of $x$, and set 
$x^*:= \ov{x}^t$ and $\hat{x}:= (\ov{x}^t)^{-1}$. Moreover, we write 
$\diag(a_1, a_2, \dots, a_n)$ for the $n\times n$-matrix with $a_1, a_2, 
\dots a_n$ on the diagonal and all the off-diagonal entries equal to zero.

To the extension $F/\bfQ$ we associate the quasi-split unitary group
$$\U(n,n):= \{ g \in \Res_{\OF/\bfZ} \GL_{2n/\OF} \mid g J_n g^* = 
J_n\},$$
where $J_n=\bsmat & -I_n \\
I_n\esmat$ and $I_n$ stands for the $n \times n$ identity matrix. For $q 
\in \Res_{\OF/\bfZ}\GL_{n/\OF}$, we set $p_q:= \bsmat q \\
& \hat{q} \esmat \in \U(n,n)$. We will also write $G$ for $\U(2,2)$ and 
$J$ for $J_2$.

For a prime ideal $\fp$ of $F$, let $F_{\fp}$ denote the completion of $F$ 
with respect to the valuation induced by $\fp$. Set $F_p:= F 
\otimes_{\bfQ} \bfQ_p$ and $\Oo_{F,p}= \OF \otimes_{\bfZ} \bfZ_p$. Note 
that if 
$p$ is 
inert or ramified in $F$, then $F_p/\bfQ_p$ is a degree two extension of 
local fields and $a \mapsto \ov{a}$ induces the non-trivial automorphism 
in $\Gal(F_p/\bfQ_p)$, while if $p$ splits in $F$, 
then $F_p \cong \bfQ_p \times \bfQ_p$, and $a \mapsto \ov{a}$ corresponds 
on the right-hand side to the automorphism defined by $(a,b) \mapsto 
(b,a)$. We 
denote the isomorphism $\bfQ_p \times 
\bfQ_p \xrightarrow{\sim} F_p$ by $\iota_p$. For a matrix $g=(g_{ij})$ 
with entries 
in $\bfQ_p \times \bfQ_p$ we also 
set $\iota_p(g)= ((\iota_p(g_{ij}))$. For a split prime $p$ the map 
$\iota_p^{-1}$ identifies $G(\bfQ_p)$ with $$G_p =
\{(g_1,g_2) \in \GL_{4}(\bfQ_p) 
\times \GL_{4}(\bfQ_p) \mid g_1 J g_2^t = J \}.$$ Note that the map 
$(g_1,g_2) \mapsto g_1$ gives a (non-canonical) isomorphism $G(\bfQ_p) 
\cong \GL_4(\bfQ_p)$.

Denote by $\AQ$ (resp. $\AQf$, $\AF$, $\AFf$) the ring of adeles of 
$\bfQ$ (resp. finite adeles of $\bfQ$, adeles of $F$, finite adeles of 
$F$). For an adele $a$, we write $a_{\tuf}$ for its finite part. 
Set $\hat{\bfZ}:= 
\prod_p \bfZ_p$, $\hat{\Oo}_F:= \OF 
\otimes_{\bfZ} \hat{\bfZ}$, $K'= \GL_2(\hat{\Oo}_F)$ and 
$K=G(\hat{\bfZ})$. Note that $K$ is a maximal compact subgroup of 
$G(\AQf)$. For any rational prime $p$, we denote by $j_p$ the canonical 
embedding $G(\bfQ_p) \hookrightarrow G(\AQ).$

Let $M_n$ denote the additive group of $n \times n$ matrices.
Set $$S=\{ h \in \Res_{\OF/\bfZ} M_{2/\OF} \mid h^* = h\}, $$ $$T=\{h \in 
S(\bfQ) 
\mid \tr (S(\bfZ)h) \subset \bfZ \},$$ 
$$T_p=\{h \in
S(\bfQ_p)
\mid \tr (S(\bfZ_p)h) \subset \bfZ_p \},$$ and $$T_{\AQ} = \prod_p 
T_p\subset M_2(\AQf).$$ 
For a matrix $h=(h_p)_p \in T_{\AQ}$, 
and a prime $p$, set $$\epsilon_p(h) = \max \{n \in \bfZ \mid 
\frac{1}{p^n} h_p \in T_p \}$$ and $$\epsilon(h) = \prod_p 
p^{\epsilon_p(h)}.$$
Note that since $h \in T_{\AQ}$, we have $\epsilon_p(h) \geq 
0$ for every $p$.

\section{Some coset decompositions}

To shorten notation in this section we put $U_n= U(n,n)$. It is well-known 
(see e.g., \cite{Bump97}, Theorem 3.3.1) that for any 
finite subset $\mB$ of $\GL_n(\AFf)$ of cardinality $h_F$ with the 
property that the canonical homomorphism $c_F$ defined as the 
composite $\AF^{\times} \twoheadrightarrow 
\AF^{\times} / F^{\times} \bfC^{\times} \hat{\Oo}_F^{\times} \cong 
\Cl_F$ restricted to $\det \mB$ is a bijection, the following 
decomposition 
holds $$\GL_n(\AF) = \bigsqcup_{b \in \mB} \GL_n(F) \GL_n(\bfC) b 
\GL_n(\hat{\Oo}_F).$$
We will call any such $\mB$ a \textit{base}. We always assume that a base 
comes with a fixed ordering, so in particular if we consider a tuple 
$(f_b)_{b \in \mB}$ 
indexed by elements of $\mB$, and apply a permutation $\sigma$ to the 
elements 
$f_b$, we do not consider the tuples $(f_b)_{b \in \mB}$ and 
$(f_{\sigma^{-1}(b)})_{b \in 
\mB}$ to be the same.

Put $$H:= \{ a
\in \Res_{\OF/\bfZ} \bfG_{m/\OF} \mid a \ov{a} = 1\},$$ and let $\varphi:
\Res_{\OF/\bfZ} \bfG_{m/\OF}
\rightarrow H$ be defined by $\varphi(a)=a/\ov{a}$. Here $\bfG_m$ denotes 
the multiplicative group.

\begin{lemma} \label{Hlemma} For every rational prime $p$, we have $\det
U_n(\bfZ_p) =
\varphi(\Oo_{F,p}^{\times})$. \end{lemma}

\begin{proof} We have $\det U_n(\bfZ_p)  \subset 
\varphi(\Oo_{F,p}^{\times})$ by
\cite{Shimura97}, Lemma 5.11(4). We
will show the other containment. Every element of 
$\varphi(\Oo_{F,p}^{\times})$ 
can be
written as $a/\ov{a}$
for some $a \in \Oo_{F,p}^{\times}$. For such an $a$,
set $x_a:=
\diag(a,1,1, \dots, \ov{a}^{-1},1,1,\dots)$, where $\ov{a}^{-1}$ is on the
$(n+1)$st position. We certainly have $x_a \in U_n(\bfZ_p)$ and 
$\det 
x_a =
a/\ov{a}$. \end{proof}

\begin{prop} \label{decomp657} Assume $2 \nmid h_F$. For any base 
$\mB$ the following
decomposition holds
\be \label{Undecomp} U_n(\AQ) = \bigsqcup_{b
\in \mB} U_n(\bfQ) U_n(\bfR) p_b U_n(\hat{\bfZ}).\ee \end{prop}

\begin{proof} First note that $\{ c_F(\det p_b)\}_{b
\in \mB}= \Cl_F$ as long as $2 \nmid h_F$. [Indeed, $\det p_q = 
\det q/
\ov{\det q}$. Since $N_{F / \bfQ} \det q \in
\AQ^{\times} \subset F^{\times} \bfC^{\times} \hat{\Oo}_F^{\times}$, i.e.,
$\det q \hs \ov{\det q}$ represents a principal fractional ideal of $F$, 
we
have
$c_F(\ov{\det q}) = c_F(\det q^{-1})$ and hence $c_F(\det 
p_b) =
c_F(\det b)^2$. Since $2 \nmid h_F$, $\alpha 
\mapsto
\alpha^2$ is an
automorphism of $\Cl_F$. Hence $c_F(\det \mB) = \Cl_F$ implies $\{ 
c_F(\det 
p_b)
\}_{b \in \mB} = \Cl_F$.] In the notation of Lemma
8.14 of \cite{Shimura97}, we have $D= U_n(\bfR) 
U_n(\hat{\bfZ})$. By
part (3) of that lemma, the map $x \mapsto \det x$ defines a bijection
between $U_n(\bfQ)\setminus U_n(\AQ)/D$ and $H(\AQ)/H(\bfQ) \det D$. 
Set
$U_0$ (in the notation of part (4) of that lemma) to be 
$\bfC^{\times} \hat{\Oo}_F^{\times}$.
Then part (4) of that
lemma asserts that $\varphi$ gives an isomorphism of 
$\AF^{\times}/\AQ^{\times}
F^{\times} U_0$ with $H(\AQ)/H(\bfQ) \varphi(U_0)$. By Lemma \ref{Hlemma} 
we
have $\varphi(U_0)= \det D$. Thus the composite of $x
\mapsto \det x$ with $\varphi$ gives a bijection of $U_n(\bfQ)\setminus
U_n(\AQ)/D$ onto $\AF^{\times}/\AQ^{\times}
F^{\times} U_0$, which in turn can be identified with $\Cl_F$ since
$\AQ^{\times} =
\bfQ^{\times} \times \bfR_+ \times \hat{\bfZ}^{\times} \subset F^{\times}
U_0$.
\end{proof}

\begin{cor} \label{scalarcor} If $(h_F, 2n)=1$ a base $\mB$ can be 
chosen
so that for all $b \in \mB$ the matrices $b$ and $p_b$ are
scalar matrices and $b b^* = b^* b = I_n$. \end{cor}

\begin{proof} It follows from the Tchebotarev density theorem, that
elements of $\Cl_F$ can be represented by prime ideals. Since all 
the inert
ideals are principal, $\Cl_F$ can be represented by prime ideals
lying over split primes of the form $(p)=\fp \ov{\fp}$. Let $\Sigma$ be a
representing
set consisting of such ideals $\fp$. As $(2n,h_F)=1$, the set 
$\Sigma^{2n}$
consisting of elements of $\Sigma$
raised to the power $2n$ is also a representing set for 
$\Cl_F$. Moreover, as
$\fp \ov{\fp}$ is a principal ideal, $\ov{\fp} = \fp^{-1}$ as
elements of $\Cl_F$, hence $\Sigma':=\{ \fp^n \ov{\fp}^{-n}\}_{\fp\in 
\Sigma}$ 
also represents all the elements $\Cl_F$. Elements
of $\Sigma'$ can be written adelically as $\alpha_{\fp}^n$, 
with $\alpha_{\fp}=(1,1, \dots, 
1,
p, p^{-1}, 1,
\dots) \in \AFf$, where $p$ appears on the $\fp$-th place and $p^{-1}$
appears at the $\ov{\fp}$-th place. Set $b_{\fp}= \alpha_{\fp} I_n$.
Then we can take $\mB=\{ b_{\fp}\}_{\fp^n \ov{\fp}^{-n} \in \Sigma'}$ 
and we have $p_{b_{\fp}}=\alpha_{\fp}
I_{2n}$. It is also clear that $b b^*=b^* b =I_n$.  
\end{proof}

\section{Hermitian modular forms}

From now on let $k$ be a positive integer divisible by $\# 
\OF^{\times}$. Let $\bfi = \bsmat i \\ & i \esmat $ and set $$\mH:= 
\{Z \in M_2(\bfC) \mid -\bfi (Z-Z^*) >0\}.$$ The group 
$G(\bfR)$ acts on $\mH$ - an element $g=\bsmat A&B \\ C&D \esmat \in 
G(\bfR)$ 
(with $A, B, C, D \in M_2(\bfC)$) sends $Z 
\in \mH$ to $gZ:=(AZ+B)(CZ+D)^{-1}$. Set $j(g,Z):= \det(CZ+D)$. For a 
congruence 
subgroup $\G \subset 
G(\bfQ)$, define $\mM'_k(\Gamma)$ to be the $\bfC$-space consisting of 
functions $\phi: \mH \rightarrow \bfC$ satisfying 
$$(\phi|_k \g)(Z):= j(\g,Z)^{-k}\phi(\g Z) = \phi(Z)$$ for all $\g \in 
\G.$
Denote by 
$\mM_k(\G)$ the subspace of 
$\mM'_k(\Gamma)$ consisting of holomorphic functions. We call elements of 
$\mM_k(\Gamma)$ the \textit{$\Gamma$-hermitian modular forms (of weight 
$k$)}.

Let $\mM'_k$ denote the $\bfC$-space consisting of functions $f 
: G(\AQ) 
\rightarrow \bfC$ satisfying the following conditions: \begin{itemize}
\item $f (\g g) = f(g)$ for all $\g \in G(\bfQ)$, $g \in G(\AQ)$,
\item $f (g \kappa) = f(g)$ for all $\kappa \in K$, $g \in 
G(\AQ)$,
\item $f(hg) = j(h_{\iy}, g_{\iy} \bfi)f(g)$ for all $g = \g g_{\iy} 
\kappa 
\in G(\bfQ) G(\bfR) K$, $h=(h_{\iy},1) \in G(\bfR) G(\AQf)$. \end{itemize} 

Let $\mC \subset G(\AQf)$ be a finite subset such that $$G(\AQ) = 
\bigsqcup_{c 
\in \mC} G(\bfQ) G(\bfR) c K.$$ In particular if $h_F$ is odd we can 
take $\mC = \{p_b\}_{b \in \mB}$ for any base $\mB$. Let $f \in \mM'_k$. 
For $g\in 
G(\AQ)$, write $g = \g g_{\iy} c \kappa \in G(\bfQ) G(\bfR) c K$ for a 
unique $c\in \mC$ and set $Z:= g_{\iy} \bfi$. Set $f_c(Z) = 
j(g_{\iy}, \bfi)^{k}f(cg)$. The 
map 
$f 
\mapsto (f_c)_{c \in \mC}$ defines a $\bfC$-linear isomorphism 
$\Phi_{\mC}: \mM'_k \xrightarrow{\sim} \prod_{c \in \mC} \mM'_k(\G_c),$ 
where 
$\G_c:= G(\bfQ) \cap 
(G(\bfR) cK c^{-1})$ (cf. \cite{Shimura97}, p. 80). If $h_F$ is odd, $\mB$ 
is a base 
and $\mC = \{p_b\}_{b \in \mB}$, we write $\G_b$ instead of $\G_{p_b}$ and 
$f_b$ instead of $f_{p_b}$ for $b \in \mB$, and $\Phi_{\mB}$ instead of 
$\Phi_{\mC}$.

\begin{definition} A function $f \in \mM'_k$ whose image under the  
isomorphism $\Phi$ lands in $\prod_{c \in \mC} \mM_k(\G_c)$ will be called 
a 
\textit{hermitian modular form of weight $k$}. The 
space of 
hermitian modular forms of weight $k$ will be denoted by $\mM_k$. 
\end{definition} We clearly have \be \label{prod1} \mM_k \cong \prod_{c 
\in \mC} 
\mM_k(\G_c).\ee

For every $q \in G(\AQf)$, $f_q \in \mM_k(\G_q)$ possesses a Fourier 
expansion $$f_q(Z) = 
\sum_{h 
\in S(\bfQ)} 
c_q(h) e^{2 \pi i \tr (hZ)}.$$ Similarly, every $f \in \mM_k$ possesses a 
Fourier expansion, i.e., for every $q \in
\GL_2(\AF)$, and every $h \in S(\bfQ)$ there exists a complex number 
$c_f(h,q)$ such that one has $$f\left( \bmat I_2 & 
\sigma \\ & I_2 \emat \bmat q \\ & \hat{q} \emat
\right) = \sum_{h \in S(\bfQ)} c_f(h,q) e_{\AQ} (\tr h \sigma)$$ for every 
$\sigma \in S(\AQ)$. Here $e_{\AQ}$ is defined in the following way. 
Let $a = (a_v) \in \AQ$, where $v$ runs over all the places of $\bfQ$. If 
$v=\iy$, set $e_v(a_v) = e^{2 \pi i a_v}$. If $v =p$, set $e_v(a_v) = 
e^{-2 \pi i y}$, where $y$ is a rational number such that $a_v-y \in 
\bfZ_p$. Then we set 
$e_{\AQ}(a) = \prod_v e_v(a_v)$. 

\begin{definition} Let $\mB$ be a base. We will say that $q \in 
\GL_2(\AF)$ belongs to a class $b \in \mB$ if there exist $\g \in 
\GL_2(F)$, $q_{\iy} \in \GL_2(\bfC)$ and $\kappa\in K'$ such that $q = \g 
b q_{\iy} \kappa$. \end{definition}

\begin{rem} It is clear that the class of $q$ depends only on $q_{\bff}$.
\end{rem}

\begin{lemma} \label{sameclass} Suppose $r \in \GL_2(\AF)$ and $q \in
\GL_2(\AF)$ belong to the same class and $r_{\bff} = \g q_{\bff}
\kappa \in \GL_2(F) q_{\bff} \GL_2(\hat{\Oo}_F)$. Then
\be \label{fourierrel} c_f(h,r) = \left( \ov{\det r_{\iy}} / 
\ov{\det q_{\iy}}\right)^{k }e^{-2 \pi \tr (r_{\iy}^* h r_{\iy} -
q_{\iy}^*
\g^* h \g q_{\iy})} (\det \g^*)^{-k} c_f(\g^* h \g, q). \ee \end{lemma}

\begin{proof} It follows from the proof of part (4) of
Proposition 18.3 of \cite{Shimura97}, that \be \label{f12} c_f(h,r) =
\left(
\ov{\det
r_{\iy}}\right)^k e^{-2 \pi \tr (r_{\iy}^* h r_{\iy})} 
c_{p_{r_{\tuf}}}(h),\ee
where
$$f_{p_{r_{\tuf}}}(Z) = \sum_{h \in S} c_{p_{r_{\tuf}}}(h) e^{2 \pi i \tr 
hZ}.$$ As is easy
to see (cf. for example the Proof of Lemma 10.8 in \cite{Shimura97}),
$f_{p_{r_{\tuf}}} = f_{p_{q_{\tuf}}} |_k \bsmat \g^{-1} \\ & \g^* \esmat$. 
Hence
\be \label{f13} c_{p_{r_{\tuf}}}(h) = (\det \g^*)^{-k}
c_{p_{q_{\tuf}}}(\g^* h
\g).\ee The Lemma follows from combining (\ref{f12}) with (\ref{f13}).
\end{proof}

\section{The Maass space} \label{The Maass space}

\begin{definition} \label{Maass form} Let $\mB$ be a base. We say that $f\in 
\mM_k$ is a $\mB$-\textit{Maass form} if there exist functions $c_{b, 
f}: 
\bfZ_{\geq 0} \rightarrow \bfC$, $b \in \mB$, such that for every $q \in 
\GL_2(\AF)$ and every $h \in S(\bfQ)$ the Fourier coefficient $c_f(h,q)$ 
satisfies
\begin{multline} \label{Maass condition} c_f(h,q) = 
\left(\ov{\det q_{\iy}}\right)^k 
e^{-2\pi\tr(q_{\iy}^* h q_{\iy})} \left( \det \g_{b, q}^*\right)^{-k} 
\times
\\
\times \sum_{\substack{d \in \bfZ_{+} \\ d \mid \epsilon(q_{\bff}^* h 
q_{\bff})}} 
d^{k-1} c_{b, f}\left( D_F d^{-2} \det h \hs \prod_{p} p^{\val_p(\det 
q_{\bff}^*q_{\bff})}\right),\end{multline} where $q_{\bff} = \g_{b, 
q} b 
\kappa_q \in \GL_2(F) b K'$ for a unique $b 
\in \mB$. \end{definition} 

\begin{rem} Note that by \cite{Shimura97}, Proposition 18.3(2), $c_f(h,q) 
\neq 0$ only if $(q^* h q)_p \in T_p$, so $\epsilon_p(q_{\bff}^* h 
q_{\bff})\geq 0$. Also, note that Definition \ref{Maass form} is 
independent of the decomposition $q_{\bff} = \g_{b,   
q} b
\kappa_q\in \GL_2(F) b K'$. Indeed, 
if $q_{\bff} = \g'_{b,
q} b
\kappa'_q \in \GL_2(F) b K'$ is 
another decomposition of $q_{\bff}$, then 
$$\det \g'_{b,
q} \det \g_{b,
q}^{-1} = \det (\kappa_q (\kappa'_q)^{-1})\in \hat{\Oo}_F^{\times} \cap 
F^{\times} =
\OF^{\times},$$ 
so $\det (\g'_{b,
q})^k =  \det \g_{b,
q}^k$. \end{rem}

\begin{definition} \label{Maass space} The $\bfC$-subspace of $\mM_k$ 
consisting of $\mB$-Maass forms 
will be called the $\mB$-\textit{Maass space}. \end{definition}

Let $T'_{\AQ}$ (resp. $T'_p$) be defined in the same way as $T_{\AQ}$ 
(resp. $T_p$) except we do not require that $h \in T'_{\AQ}$ (resp. 
$T'_p$) be hermitian. Define $\epsilon'$, $\epsilon'_p$ in the same way 
as $\epsilon$, $\epsilon_p$ except replace $T$ with $T'$. Then 
$\epsilon = \epsilon'|_{T}$. \begin{lemma} 
\label{epsilon1} If $h=(h_p)_p \in T'_{\AQ}$ and $\kappa=(\kappa_p)_p \in
K'$, then $$\epsilon'(h \kappa) = \epsilon'(\kappa h) = \epsilon'(h).$$
\end{lemma}

\begin{proof} If $h_p = p^n h'$, $h' \in T'_p$, then $\kappa_p h_p = p^n 
\kappa_p
h' \in p^n T'_p$. On the other hand, if $\kappa_p h_p \in p^n T'_p$, then 
$h_p =
\kappa_p^{-1} (\kappa_p h_p) \in p^n T'_p$ by the above argument. So, $h_p 
\in 
p^n
T'_p$ if and only if $\kappa_p h_p \in p^n T'_p$, so $\epsilon_p(h \kappa) 
=
\epsilon_p(h)$. The other equality is proved in the same way. \end{proof}

\begin{cor} \label{epsilon} If $q \in \GL_2(\AF)$ and $r \in
\GL_2(\AF)$ are in the same class, and $r_{\bff} =\g q_{\bff} \kappa
\in \GL_2(F)
q_{\bff}K'$, then $\epsilon(r_{\bff}^* h
r_{\bff}) = \epsilon( q_{\bff}^* \g^* h
\g q_{\bff})$. \end{cor}

\begin{prop} \label{to check} Choose a base $\mB$ and let $f \in \mM_k$. 
If 
there exist functions $c^*_{b,f}: \bfZ_{\geq 0} \rightarrow \bfC$, $b \in 
\mB$, such 
that for every $b \in \mB$ and every $h \in S(\bfQ)$, the Fourier coefficient 
$c_f(h,b)$ satisfies 
condition (\ref{Maass condition}) with $c_{b,f}^*$ in place of $c_{b,f}$, 
then $f$ is a $\mB$-Maass form and one has $c_{b,f}=c^*_{b,f}$ for every 
$b \in \mB$. \end{prop}

\begin{proof} Fix $\mB$ and $f \in \mM_k$. Suppose there exist 
$c_{b,f}^*$ 
such that (\ref{Maass condition}) is satisfied for all pairs $(h,b)$. Let 
$q = \g b x \kappa = (\g x, \g b \kappa)\in \GL_2(\bfC) \times 
\GL_2(\AFf)$, where $\g \in \GL_2(F)$, $x \in \GL_2(\bfC)$ and $\kappa 
\in K'$. Then by Lemma \ref{sameclass}, $$c_f(h,q) = \left( \ov{\det 
q_{\iy}}\right)^k e^{-2 \pi \tr (q_{\iy}^* h q_{\iy} -
\g^* h \g )} (\det \g^*)^{-k} c_f(\g^* h \g, b).$$ Since condition 
(\ref{Maass 
condition}) is satisfied for $(h,b)$, we know that
$$c_f(h,b)= e^{-2 \pi \tr h} 
\sum_{\substack{d \in \bfZ_{+} \\ d \mid \epsilon(b^* hb)}}
d^{k-1} c^*_{b, f}\left( D_F d^{-2} \det h \hs \prod_{p} p^{\val_p(\det   
b^*b)}\right).$$ Thus $$c_f(\g^*  h \g,b) = e^{-2 \pi \tr(\g^* h \g)} 
\sum_{\substack{d \in \bfZ_{+} \\ d \mid \epsilon(b^*\g^* h\g b)}}
d^{k-1} c^*_{b, f}\left( D_F d^{-2} \det (\g^* h\g) \hs \prod_{p} 
p^{\val_p(\det
b^*b)}\right).$$ So, \begin{multline} c_f(h,q) = \left(\ov{\det 
q_{\iy}}\right)^k
e^{-2\pi\tr(q_{\iy}^* h q_{\iy})} \left( \det \g^*\right)^{-k}
\times\\
\times \sum_{\substack{d \in \bfZ_{+} \\ d \mid \epsilon(b^* \g^* h
\g b)}}
d^{k-1} c^*_{b, f}\left( D_F d^{-2} \det h \det(\g^* \g) \hs \prod_{p} 
p^{\val_p(\det
b^* b)}\right).\end{multline}
The claim now follows since $\epsilon(b^* \g^* h   
\g b) = \epsilon(q_{\bff}^* h q_{\bff})$ by Corollary \ref{epsilon} and $\det 
(\g^* \g) \in \bfQ_+$, so $\det (\g^* \g) = \prod_p p^{\val_p(\det \g^* 
\g)}$. 
\end{proof}

\begin{prop} \label{independence} If $\mB$ and $\mB'$ are two bases, then 
the $\mB$-Maass space and the $\mB'$-Maass space coincide, i.e., the 
notion of a Maass form is independent of the choice of the base. 
\end{prop}

\begin{proof} Let $\mB$ and $\mB'$ be 
two bases. Write $q_{\bff} =  
\g_{b,q} b
\kappa_{\mB}= \g_{b',q} b' \kappa_{\mB'}$ with $b \in \mB$, 
$b'\in \mB'$, $\g_{b,q}, 
\g_{b',q} \in \GL_2(F)$ and $\kappa_{\mB},  \kappa_{\mB'} \in K'$. 
Suppose $f$ is a $\mB$-Maass form, i.e., 
there 
exist functions $c_{b,f}$ for $b \in \mB$, such that for every $q$ and 
$h$,  
$$c_f(h,q) = t \hs \det(\g_{b,q}^*)^{-k} 
\sum_{\substack{d \in \bfZ_{+} \\ d \mid \epsilon(q_{\bff}^* h
q_{\bff})}}
d^{k-1} c_{b, f}(s),$$ where $t=\left(\ov{\det q_{\iy}}\right)^k
e^{-2\pi\tr(q_{\iy}^* h q_{\iy})} $ and 
$s= D_F d^{-2} \det h \hs 
\prod_{p} p^{\val_p(\det
q_{\bff}^*q_{\bff})}$. Our goal is to show that there exist 
functions $c_{b',f}$ for $b' \in \mB'$, 
such that for every $q$ and $h$, \be \label{want1} c_f(h,q) = t \hs 
\det(\g_{b',q}^*)^{-k}
\sum_{\substack{d \in \bfZ_{+} \\ d \mid \epsilon(q_{\bff}^* h
q_{\bff})}}
d^{k-1} c_{b', f}(s).\ee We have \be \label{dets1} \det \g_{b,q} = 
\det 
\g_{b',q} 
\det (b' b^{-1}) \det (\kappa_{\mB}^{-1} \kappa_{\mB'}).\ee
Since $\det(b'b^{-1})$ corresponds to a principal fractional ideal, say 
$(\alpha_{b,b'})$, under the map $((\alpha_{b,b'})_{\fp}) \mapsto 
\prod_{\fp} 
\fp^{\val_{\fp}((\alpha_{b,b'})_{\fp})}$, using \cite{Bump97}, Theorem 
3.3.1, we 
can write $\det(b' b^{-1}) = \alpha_{b,b'} \kappa_{b,b'}\in 
\AFf^{\times}$ with 
$\kappa_{b,b'} \in \hat{\Oo}_F^{\times}$. Then it follows from (\ref{dets1}) 
that 
$$\beta:= \kappa_{b,b'} \det(\kappa_{\mB}^{-1} \kappa_{\mB'}) \in 
\hat{\Oo}_F^{\times} \cap F^{\times} = \OF^{\times}.$$ 
Hence $\beta^k=1$. Thus $\left(\det \g_{b,q}^*\right)^{-k} = \left(\det 
\g_{b',q}^*\right)^{-k} \ov{\alpha_{b,b'}}^{-k}$. Note that 
$\alpha_{b,b'}^{-k}$ is 
well defined and only depends on $b$ and $b'$ (i.e., it is independent 
of $q$ and $h$). Set $c_{b', f}(n) = 
\ov{\alpha_{b,b'}}^{-k}c_{b, f}(n)$. Then it is clear that $c_{b', f}$ 
satisfies (\ref{want1}).
\end{proof} 

\begin{definition} \label{maass2} From now on we will refer to $\mB$-Maass 
forms simply 
as 
\textit{Maass forms}. Similarly we will talk about the \textit{Maass space} 
instead of $\mB$-Maass spaces. This is justified by Proposition 
\ref{independence}. The Maass space will be denoted by 
$\mM^{\tuM}_k$. \end{definition}

We now recall the definition of Maass space introduced in \cite{Krieg91}. 
We will refer to it as the $G(\bfZ)$-Maass space. Consider the space 
$\mM_k(G(\bfZ))$ of $G(\bfZ)$-hermitian modular forms of weight $k$. We 
say that $\phi
(Z)= \sum_{h \in T} c_{\phi}(h) e^{2 \pi i \tr (hZ)} \in \mM_k(G(\bfZ))$ 
is a 
\textit{$G(\bfZ)$-Maass form} if there exists a function $\alpha_{\phi}: 
\bfZ_{\geq 0} 
\rightarrow \bfC$ such that for every $h \in T$, one has \be 
\label{kriegcond} c_{\phi}(h) = \sum_{\substack{ d \in \bfZ_+\\ d \mid 
\epsilon(h)}} d^{k-1} \alpha_{\phi}(D_F d^{-2} \det h).\ee The subspace of 
$\mM_k(G(\bfZ))$ consisting of $G(\bfZ)$-Maass forms will be denoted by 
$\mM_k^{\tuM}(G(\bfZ))$.

\begin{prop} \label{oddclass} If $2 \nmid h_F$, then the Maass space 
$\mM_k^{\tuM}$ is isomorphic 
(as a 
$\bfC$-linear space) to $\# \mB$ copies of the $G(\bfZ)$-Maass space 
$\mM^{\tuM}_k(G(\bfZ))$.
\end{prop}

\begin{proof} Since the Maass space is independent of the choice of 
a base $\mB$ by Proposition \ref{independence}, we may choose $\mB$ 
as in Corollary \ref{scalarcor}. The map $\Phi_{\mB}: \mM_k 
\rightarrow \prod_{b \in 
\mB} \mM_k(G(\bfZ))$ is an isomorphism. Let $f \in \mM_k^{\tuM}$ 
and set $(f_b)_{b \in \mB}= \Phi_{\mB}(f)$. Set $\a_{f_b}:= c_{b,f}$. 
Then using
(\ref{f12}), and the fact that the matrices $b$
commute with
$h$ and
$b^*b=1$, we see that condition (\ref{Maass 
condition}) for $c_f(h,b)$ translates 
into condition (\ref{kriegcond}) for $c_{f_b}(h)$.
Hence $\Phi_{\mB}(\mM^{\tuM}_k) \subset 
\prod_{b \in \mB} \mM^{\tuM}_k(G(\bfZ))$. On the other hand if $(f_b)_{b 
\in \mB} \in \prod_{b \in \mB} \mM^{\tuM}_k(G(\bfZ))$, set $c_{b,f}:= 
\alpha_{f_b}$. Then conditions 
(\ref{kriegcond}) for $c_{f_b}(h)$ translate into conditions 
(\ref{Maass condition}) for $c_f(b,h)$. By Proposition 
\ref{to check} this implies 
that $f$ is a Maass form.
\end{proof}

\section{Invariance under Hecke action}

It was proved in \cite{Krieg91} that the $G(\bfZ)$-Maass space is 
invariant under the action of a certain Hecke operator $T_p$ associated 
with a prime $p$ 
which is inert in $F$. On the other hand Gritsenko in \cite{Gritsenko90} 
proved the invariance of the $G(\bfZ)$-Maass space under all the Hecke 
operators when the class number of $F$ is equal to 1. In this section we 
show that if the class number of $F$ is odd, then the Maass space 
$\mM_k^{\tuM}$ is 
in 
fact invariant under all the local Hecke algebras (for primes $p \nmid 
D_F$).

\subsection{The Hecke algebra}

From now on assume that $h_F$ is odd. Let $p$ be a rational prime and 
write $K_p$ for $G(\bfZ_p)$. Let $\mH_p$ 
be the $\bfC$-algebra 
generated by 
double cosets $K_p g K_p$, $g \in G(\bfQ_p)$ with the usual law of 
multiplication (cf. \cite{Shimura97}, section 11). 
If $K_pgK_p \in \mH_p$, there exists a finite set $A_g \subset 
G(\bfQ_p)$ such that $K_p g K_p = \bigsqcup_{\alpha \in A_g} K_p \alpha$. 

For $f \in \mM_k$, $g \in G(\bfQ_p)$, $h \in G(\AQ)$, set 
$$([K_p g K_p]f)(h) = \sum_{\a 
\in A_g } f(h j_p(\a)^{-1}).$$ It is clear that $[K_p g K_p]f \in \mM_k$.

\begin{thm} \label{thmmain} Let $p \nmid D_F$ be a rational prime. The 
Maass 
space is invariant under the action of $\mH_p$, i.e., if $f \in 
\mM^{\tuM}_k$, and $g \in G(\bfQ_p)$, then $[K_pgK_p]f \in \mM^{\tuM}_k$. 
\end{thm}

\begin{proof} We will only present the proof in the case when $p$ 
splits in $F/\bfQ$. For such a prime $p$ the invariance of 
$\mM_k^{\tuM}$ under the action of $\mH_p$ follows from Lemma
\ref{generation} and Propositions
\ref{invariance}, \ref{invariance2} and \ref{invariance3} below. If $p$ is 
inert one 
can proceed along the same lines, 
however, it is the case when $p$ splits that is essentially new.  
Indeed, if $p$ is inert, the elements of
$\mH_p$ respect the decomposition
(\ref{prod1}), hence the statement of the theorem reduces to an assertion 
about the action of $\mH_p$ on $\mM_k(G(\bfZ))$. Then 
the method used in \cite{Gritsenko90} can be adapted to prove the theorem. 
See 
also Theorem 7 in \cite{Krieg91} which proves the invariance of the 
$G(\bfZ)$-Maass space for a certain family of Hecke operators in $\mH_p$.
\end{proof}

Let $p$ be a prime which splits in $F$. Write $(p)=\fp \ov{\fp}$. 
Recall that $G(\bfQ_p) \cong \GL_4(\bfQ_p)$, and
an element 
$A$
of $G(\bfQ_p)$ can be written as $A=(A_1, A_2) \in \GL_4(\bfQ_p) \times 
\GL_4(\bfQ_p)$ with $A_2 = -J (A_1^t)^{-1} J.$ Note that if 
$a,b,c,d \in M_2(\bfC)$ then $$-J \bmat a & 
b \\ c &d \emat J = \bmat d & -c \\ -b 
&a \emat.$$ Set \begin{itemize} \item $T_{\fp} := K_p \iota_p((\diag( 
p^{-1}, 1,1,1), \diag(1,1,p,1) ))K_p$
\item $U_{\fp}:= K_p \iota_p((\diag (p^{-1}, p^{-1}, 1,1) , 
\diag(1,1,p,p)))K_p,$
\item $\Delta_{\fp} := K_p \iota_p((pI_4,p^{-1}I_4)) K_p.$ \end{itemize}

\begin{lemma} \label{generation} The $\bfC$-algebra $\mH_p$ is generated 
by the operators 
$T_{\fp}$, $T_{\ov{\fp}}$, $U_{\fp}$, $\Delta_{\fp}$ and their 
inverses. \end{lemma}

\begin{proof} This follows from the theory of Hecke algebras on 
$GL_4(\bfQ_p)$. \end{proof}

\subsection{Diagonalizing hermitian matrices mod $p^n$}

We begin by proving an analogue of Proposition 7 of \cite{Krieg91} for a 
split prime $p$.

\begin{prop} \label{diagonal} Let $n$ be a positive integer, $p$ a 
prime number split in $F$ and 
write $(p) = \fp \ov{\fp}$. Let $h \in T$, $h \neq 0$. Then there exist 
$a$, $d \in \bfZ_+$ with $p \nmid a$ and $u \in \SL_2(\OF)$ such that 
$$u^* h u \equiv \epsilon(h) \bmat a \\ &d \emat \pmod{p^nT}.$$ 
\end{prop}

In fact it is enough to prove the following lemma.

\begin{lemma} \label{diagonal2} Proposition \ref{diagonal} holds if $T$ is 
replaced by the subgroup of hermitian matrices inside $M_2(\OF)$. 
\end{lemma}

\begin{proof} Without loss of generality we may assume that 
$\epsilon(h)=1$. Let $(M, M^t)$ be the image of $h$ under the composite 
$$M_2(\OF) \twoheadrightarrow M_2(\OF/p^n) \xrightarrow{\sim} 
M_2(\OF/\fp^n) \oplus M_2(\OF/\ov{\fp}^n) \cong M_2(\bfZ/p^n) \oplus 
M_2(\bfZ/p^n).$$ Since the canonical map $\SL_2(\OF) \rightarrow 
\SL_2(\OF/p^n) \cong \SL_2(\bfZ/p^n) \oplus \SL_2(\bfZ/p^n )$ is surjective 
(\cite{Serre70}, p. 490), it is enough to find $A_1, A_2 \in 
\SL_2(\bfZ/ p^n \bfZ)$ such that \be \label{conj1} A_2^t M A_1 = \bsmat 
\alpha \\ 
& \delta \esmat\ee
with $\alpha \neq 0$ mod $p$. 
The existence of such $A_1$ and $A_2$ is clear.
\end{proof}

\subsection{Invariance under $T_{\fp}$}

\begin{lemma} \label{decomp1} We have the following decomposition 
\be \begin{split} T_{\fp}  = & \bigsqcup_{a,b,c \in \bfZ/p \bfZ} K_p 
\left( 
\bsmat p^{-1} & 
ap^{-1} & bp^{-1} & cp^{-1} \\ & 1\\ && 1\\ &&&1 \esmat, \bsmat 1&&b\\ &1 
& c \\ && p \\ && -a & 1 \esmat \right)  \sqcup \\
& \bigsqcup_{d,e \in \bfZ/p \bfZ}  K_p \left( \bsmat 1 \\ & p^{-1} & 
dp^{-1} 
& ep^{-1} \\ && 1 \\ &&& 1 \esmat , \bsmat 1 &&& d \\ & 1 && e \\ && 1 \\ 
&&& p \esmat \right)\sqcup \\
& \bigsqcup_{f \in \bfZ/p \bfZ} K_p \left( \bsmat 1\\ & 1 \\ && p^{-1} & 
p^{-1} f \\ &&& 1 \esmat , \bsmat p \\ -f & 1 \\ && 1 \\ &&& 1 \esmat 
\right) \sqcup \\
& K_p \left( \bsmat 1 \\ & 1 \\ && 1 \\ &&& p^{-1} \esmat, \bsmat 1 \\ & p 
\\ && 1 \\ &&& 1 \esmat \right). 
\end{split} \ee \end{lemma}

\begin{proof} This follows from an analogous decomposition for 
$\GL_4(\bfQ_p)$. 
\end{proof}

Let $g = T_{\fp} f$. Then for $q \in \GL_2(\AF)$ and $\sigma \in S(\AQ)$, 
we 
can write $$g \left( \bmat q & \sigma \hat{q} \\ & \hat{q} \emat \right) 
= \sum_{h \in S} c_g(h,q) e_{\AQ}(\tr (h \sigma)).$$ Define the following 
matrices $$\alpha'_a = \left(\bmat p & -a \\ & 1 \emat, I_2 \right) \in 
\GL_2(\bfQ_p) \times \GL_2(\bfQ_p)
\quad 
a=0,1, \dots, 
p-1$$ and $$\alpha'_p = \left( \bmat 1 \\ & p \emat, I_2 \right),\in 
\GL_2(\bfQ_p) \times \GL_2(\bfQ_p).$$ 
For $a=0,1, 
\dots, p$, set $\alpha_a= \iota_p(\alpha'_a) \in \GL_2(F_p)$.
\begin{lemma} \label{Fourier1} One has the following formula
$$c_g(h,q) = p^2 \sum_{a=0}^{p} 
 c_f(h, q \alpha_a) + \sum_{a=0}^{p} c_f(h, q \hat{\alpha}_a). 
$$\end{lemma}

\begin{proof} This is a straightforward calculation using Lemma 
\ref{decomp1}. \end{proof} 

\begin{prop} \label{invariance} The Maass space is invariant under the 
action of $T_{\fp}$, i.e., if $f\in \mM_k^{\tuM}$, then $g\in 
\mM^{\tuM}_k$. 
\end{prop}

\begin{proof} Choose a base $\mB$ as in Corollary \ref{scalarcor}. In 
particular we have $\epsilon(b^* h 
b) = \epsilon(h)$ and $\val_p (\det b^* b) = 0$. By Propositions \ref{to 
check} and \ref{independence} 
it is enough to show that there exist functions $c_{b, g}: \bfZ_+ 
\rightarrow \bfC$ ($b \in \mB$) such that \be \label{condrep} c_g(h,b) = 
e^{-2\pi\tr  h} 
\sum_{\substack{ d \in \bfZ_+ \\ d \mid \epsilon(h)}} d^{k-1} c_{b,g} 
\left( D_F d^{-2} \det h \right).\ee For $b \in \mB$, set $b' = b 
\alpha_p$. Note that all of the matrices: $b \a_a$, 
$b \hat{\a}_a$, ($a=0,1, \dots, p$) belong to the same 
class $b'$. Denote any of those matrices by $q$. Then $q = 
\gamma_{b', q} b' \kappa_q \in \GL_2(F) b' K'$ and it is easy to see that 
$$\det \gamma_{b', q}^k = \begin{cases} 1 & q = b \a_a, \hf a = 0, 1, 
\dots, p \\ p^{-k} & q = b \hat{\a}_a \hf a = 0, 1, 
\dots, p , 
\end{cases}$$ and $$\val_p(\det q^* q ) =\begin{cases} 1 & q = b 
\a_a \hf a = 0, 1, 
\dots, p \\ -1 & q = b \hat{\a}_a \hf a = 0, 1, 
\dots, p .  
\end{cases}$$ Write $h = \epsilon(h) h'$. One has $\epsilon(h') = 1$. 
Set $D= D_F \det h$ and $D' = D_F \det h'$.
Using Lemma 
\ref{Fourier1} and the fact that $f$ is a Maass form, we obtain 
\begin{multline} \label{Fourier2} c_g(h,b) = e^{-2 \pi \tr h} \times 
\left( p^2  
\sum_{a=0}^p \sum_{\substack {d \in \bfZ_+ \\ d \mid \epsilon(\a_a^* h 
\a_a)}} d^{k-1} c_{b',f}(D d^{-2} p) + \right. \\
+ \left. p^k \sum_{a=0}^p \sum_{\substack {d \in \bfZ_+ \\ d \mid 
\epsilon(\hat{\a}_a^* 
h
\hat{\a}_a)}} d^{k-1} c_{b',f}(D d^{-2} p^{-1})\right). \end{multline}
Using Proposition \ref{diagonal}, one can relate $\epsilon(\a_a^* h \a_a)$ 
and $\epsilon(\hat{\a}_a^* h \hat{\a}_a)$ to $\epsilon(h)$ for $a=0,1, 
\dots, p$, and then 
(\ref{Fourier2}) becomes \begin{multline} \label{mess} c_g(h,b)  = 
e^{-2 \pi 
\tr h} p^2 \sum_{0} A^{(1)}_d + 
e^{-2 \pi \tr h} \times \\
\times \begin{cases} p^3 \sum_0 
A^{(1)}_d 
& p \nmid D', p \nmid \epsilon(h)\\
p^k(p+1) \sum_{-1} A^{(-1)}_d + p^3 \sum_0 A_d^{(1)} & p \nmid D', p \mid 
\epsilon(h) 
\\
p^2(p-1) \sum_0 A_d^{(1)} + p^2 \sum_1 A_d^{(1)} + p^k \sum_0  A^{(-1)}_d& 
p \mid 
D', p 
\nmid 
\epsilon(h)\\
p^2(p-1) \sum_0 A_d^{(1)} + p^2 \sum_1 A_d^{(1)} + p^k \sum_0  A^{(-1)}_d 
+ 
p^{k+1}  
\sum_{-1} 
 A^{(-1)}_d 
& p \mid 
D', p \mid \epsilon(h), \end{cases} \end{multline}
where $\sum_n A_d^{(m)}= \sum_{d \mid
p^n \epsilon(h)}A_d^{(m)}$, $A^{(m)}_d = d^{k-1} c_{b',f}(D d^{-2} p^m)$.

For $D$ in the image of the map $h
\mapsto D_F \epsilon(h)^{-2} \det h$ and $b \in \mB$ we make the following
definition
\be \label{maass1} c_{b,g}(D) =  p^2(p+1 ) c_{b', f} (Dp) + p^k (p+1)
c_{b',f} (Dp^{-1}) ,\ee where we assume that
$c_{b',f}(n)=0$ when $n \not\in \bfZ_+$. If $D$ is not in the image of 
that map, we set $c_{b,g}(D)=0$. Note that we clearly have 
$$c_g(h,b) = e^{-2 \pi \tr h} c_{b,g} (D_F \det h)$$ for every $h$ with 
$\epsilon(h)=1$. Thus to check if $g$ lies in the Maass space we just need 
to check that (\ref{condrep}) holds with $c_{b,g}$ defined by  
(\ref{maass1}). 
This is an easy calculation using (\ref{mess}). \end{proof}

\subsection{Invariance under $U_{\fp}$}

This is completely analogous to the proof for $T_{\fp}$, hence we only 
include the relevant formulas for the reader's convenience.
\begin{lemma} We have the following decomposition:
\be \begin{split} U_{\fp} = & \bigsqcup_{b,c,d,e \in \bfZ/ p \bfZ} K_p 
\left( \bsmat p^{-1} & & bp^{-1} & dp^{-1} \\ & p^{-1} & cp^{-1} & ep^{-1} 
\\ && 1 \\ &&&1 \esmat , \bsmat 1 && b & c \\ & 1 & d & e \\ &&p \\ &&& p 
\esmat \right) \sqcup \\
& \bigsqcup_{a,c,f \in \bfZ/p \bfZ} K_p \left( \bsmat 1\\  
-fp^{-1} 
& p^{-1} &  cp^{-1} \\ && 1 \\ && -ap^{-1} & p^{-1} \esmat, \bsmat 1 & a 
&& c \\ &p 
\\ &&1&f \\ &&& p \esmat \right) \sqcup\\
& \bigsqcup_{e,f \in \bfZ/p\bfZ} K_p\left( \bsmat 1 \\ -fp^{-1} & p^{-1} 
&& 
ep^{-1} \\ && p^{-1} \\ &&&1 \esmat, \bsmat p \\ & 1 && e \\ && 1& f \\ 
&&& p \esmat\right) \sqcup \\
& \bigsqcup_{a,b \in \bfZ/ p \bfZ} K_p\left( \bsmat p^{-1} && bp^{-1} \\ & 
1 
\\ && 1 \\ && -ap^{-1} & p^{-1} 
\esmat , \bsmat 1 & a& b \\ & p \\ && p \\ &&&1 \esmat \right) \sqcup \\
& \bigsqcup_{d\in \bfZ/ p \bfZ} K_p\left( \bsmat p^{-1} &&& dp^{-1} \\ &1 
\\ 
&& p^{-1} \\ &&& 1 \esmat , \bsmat p\\ & 1& d \\ &&p \\ &&&1 \esmat 
\right) \sqcup \\
& K_p \left( \bsmat 1 \\ & 1 \\ && p^{-1} \\ &&& p^{-1} \esmat, \bsmat p 
\\ 
& 
p \\ && 1 \\ &&&1 \esmat \right).\end{split} \ee

\end{lemma}

As before, let $g = T_{\fp} f$. Define matrices:
$$\beta'_p = \left( \bmat p \\ & p \emat, I_2\right) \in 
\GL_2(\bfQ_p)\times 
\GL_2(\bfQ_p),$$
$$\gamma'_a = \left( \bmat 1 \\ a & p \emat, I_2\right) \in \GL_2(\bfQ_p) 
\times \GL_2(\bfQ_p),\quad a=0,1,\dots, p-1,$$
$$\gamma'_p = \left( \bmat p \\ & 1 \emat, I_2 \right) \in \GL_2(\bfQ_p)
\times \GL_2(\bfQ_p),$$ and set $\beta_p = \iota_p(\beta'_p)\in 
\GL_2(F_p)$, $\gamma_a= \iota_p(\gamma'_a)\in \GL_2(F_p)$ ($a=0,1, \dots, 
p$). 

\begin{lemma} One has the following formula
$$c_g(h,q) = p^4 c_f(h,q \beta_p) + c_f(h, q \hat{\beta}_p) + p 
\sum_{a=0}^p \sum_{b=0}^p c_f(h, q \gamma_a 
\hat{\gamma}_b).$$ \end{lemma}

\begin{prop} \label{invariance2} The Maass space is invariant under the 
action of $U_{\fp}$. \end{prop}

\begin{proof} This is similar to the proof of Proposition 
\ref{invariance}. Let $b$, $D$, $D'$, $h$, $h'$ be as in that proof. Then 
for all 
$a$, $c$, we 
see that $b\beta_p$, $b \hat{\beta}_p$, $b \gamma_a \hat{\gamma}_c$ all 
lie in the same class $b'= b \beta_p$. 
One has $$\det 
\gamma_{b', q}^k = \begin{cases} 1 & q = b \beta_p \\ p^{-k} & q = b 
\gamma_a \hat{\gamma}_c, \hf a,c \in \{ 0,1, \dots, p \} \\ p^{-2k} & q = 
b \hat{\beta}_p, \end{cases}$$ and $$\val_p (\det q^* q) = \begin{cases} 2 
& q = b \beta_p \\ 0 & q = b
\gamma_a \hat{\gamma}_c, \hf a,c \in \{ 0,1, \dots, p \} \\ -2 & q =
b \hat{\beta}_p. \end{cases}$$

Using Proposition \ref{diagonal} as in the proof of Proposition 
\ref{invariance} we obtain
\begin{multline} c_g(h,b) = e^{-2 \pi \tr h} p^4 \sum_1A^{(2)}_d + p^{2k} 
\sum_{-1} A^{(-2)}_d + p^{k+1}(p+1) \sum_0 A^{(0)}_d + p^{k+3} \sum_{-1} 
A^{(0)}_d+\\
+ e^{-2 \pi \tr h} p^{k+1} \begin{cases} p \sum_{-1} A^{(0)}_d & p 
\nmid D'\\
p \sum_0 A^{(0)}_d & p \mid D', \hs p^2 \nmid D'\\
\sum_1 A^{(0)} + (p-1) \sum_0 A^{(0)}& p^2 \mid D',\end{cases} 
\end{multline}
where if there is no $d$ dividing $p^n \epsilon(h)$, we set $\sum_n = 0$.
For $D$ in the image of the map $h \mapsto D_F \epsilon(h)^{-2} \det h$, 
we make the following definition: \begin{multline} \label{for1}  
c_{b,g}(D) = 
p^4c_{b',f}(Dp^2) + (p^{k+3}+p^{k+2} + p^{k+1}) c_{b',f}(D) + \\
+ \begin{cases} 0 & p \nmid D \\
p^{k+2} c_{b'f}(D) & p \mid D, \hs p^2 \nmid D \\
p^{k+2} c_{b',f}(D) + p^{2k} c_{b',f}(Dp^{-2}) & p^2 \mid D. \end{cases} 
\end{multline} We now check as in the proof of Proposition 
\ref{invariance} that $g$ is a Maass form. 
\end{proof}

\subsection{Invariance under $\Delta_{\fp}$}

Let $g = \Delta_{\fp}f $ and set $\delta_p = \iota_p((p^{-1}I_4, p I_4))$. 
Then we have $$c_g(h,q) = c_f(h, q \delta_p).$$
\begin{prop} \label{invariance3} The Maass space is invariant under the 
action of $\Delta_p$. \end{prop}

\begin{proof}
Let $\mB$ be as before. For $b \in \mB$ set $b' = b 
\delta_p$. One clearly has $\val_p( \delta_p^* \delta_p) = 0$ and 
$\epsilon(\delta_p^* h \delta_p) = \epsilon(h)$. Hence one can define 
\be \label{for2} c_{b,g}(D) = 
c_{b',f}(D).\ee The claim is now clear. \end{proof}

\section{Descent} \label{Descent}

Assume that $h_F$ is odd. Let $\chi_F$ be the quadratic Dirichlet 
character attached to the extension $F/\bfQ$. For a positive integer $n$, 
set $$a_F(n) = 
\# \{ \alpha \in (i D_F^{-1/2} \OF)/\OF \mid D_F N_{F/\bfQ}(\alpha) 
\equiv -n \pmod{D_F}\}.$$ Let $\mB$ be a base as in Corollary 
\ref{scalarcor}.

\begin{thm} \label{desc4} There exists a $\bfC$-linear injection of vector 
spaces 
$$\Desc_{\mB}: \mM_k^{\tuM} \hookrightarrow \prod_{b \in \mB} M_{k-1} 
(D_F, 
\chi_F),$$ 
such that $\Desc(f) = (F_b)_{b \in \mB}$ with $$a_{F_b}(n) = i 
\frac{a_{F}(n)}{\sqrt{D_F}} c_{b,f}(n),$$ where $a_{F_b}(n)$ is the 
$n$-th Fourier coefficient of $F_b$. The map $\Desc_{\mB}$ 
depends on the choice of $\mB$. \end{thm} 

\begin{proof} This follows immediately from \cite{Krieg91}, Theorem 6 and 
formula (4) using our assumption on $\mB$ and (\ref{f12}). \end{proof}

\begin{rem} Krieg in \cite{Krieg91} explicitly describes the image of the 
descent map he defines and denotes it by $G_{k-1}(D_F, \chi_F)^*$. The 
image of $\Desc$ is exactly $\prod_{b \in \mB} G_{k-1}(D_F, \chi_F)^* 
\subset \prod_{b \in \mB} M_{k-1}(D_F, \chi_F)$. 
\end{rem}

Let $S_{\mB}$ denote the group of permutations of $\mB$. For $b \in \mB$, 
let $\sigma_{\fp, n} \in S_{\mB}$ be the unique permutation such that $b 
\a_p^n$ is
in the class of $\sigma_{\fp,n}(b)$. If $A=(a_b)_{b \in \mB}$ is an 
ordered tuple
indexed by elements of $\mB$, we define $\sigma_{\fp,n}A =
(a_{\sigma_{\fp,n}^{-1}(b)})_{b \in \mB}$. Moreover, for $b \in \mB$ 
and $n \in \bfZ$, write $\g_{b,\fp,n}$ for an element of $G(\bfQ)$ such 
that 
$b
\a_p^n = \g_{b,\fp,n} \sigma_{\fp,n}(b) \kappa$ for $\kappa \in K$, and 
denote 
by $\g_{\mB,\fp,n}$ the $\mB$-tuple $(\det\g_{b,\fp,n}^*)_{b \in \mB}$.

For 
a 
rational prime $p \nmid D_F$ denote by $T_p$ the 
operator acting on $\prod_{b \in \mB} M_{k-1}(D_F, \chi_F)$ which sends 
$(F_b)_{b \in \mB}$ to $(F'_b)_{b \in \mB}$, where
$F_b(z) = 
\sum_{n=1}^{\iy} a(n) e^{2 \pi i nz}$, $F'_b(z) = 
\sum_{n=1}^{\iy} a'(n) 
e^{2 \pi i nz}$ with $a'(n) = a(np) + \chi_F(p)p^{k-2} 
a(n/p)$. Here $a(m)=0$ if $m \not\in \bfZ_{\geq 0}$. Denote by 
$\bfT_p$ 
the 
$\bfC$-subalgebra of endomorphisms of $\prod_{b \in \mB} M_{k-1}(D_F, 
\chi_F)$ generated by $T_p$, the group $S_{\mB}$ and $\mB$-tuples of 
complex numbers acting on $\prod_{b \in \mB} M_{k-1}(D_F,
\chi_F)$ in an obvious way. 

\begin{thm} \label{heckedesc} Let $p$ be a rational prime which splits in 
$F/\bfQ$. There 
exists a $\bfC$-algebra map $$\Desc_{\mB,p}: \mH_p \rightarrow 
\bfT_p,$$ such 
that for 
every $H \in \mH_p$ the 
following diagram
$$\xymatrix@C7em{\mM_k \ar[r]^{H} \ar[d]^{\Desc_{\mB}} & \mM_k
\ar[d]_{\Desc_{\mB}} \\ \prod_{b \in \mB} M_{k-1} (D_F, \chi_F)
\ar[r]^{\Desc_{\mB,p}(H)} & \prod_{b \in \mB} M_{k-1} (D_F,
\chi_F)}$$ commutes. Moreover, one has \be \begin{split} 
\Desc_{\mB,p}(T_{\fp}) & = (\gamma_{\mB,\fp,1})^{-k} p^2 (p+1) T_p \circ
\sigma_{\fp,1}, \\
\Desc_{\mB,p}(U_{\fp})& =(\gamma_{\mB,\fp,2})^{-k} p^4( T^2_p   
+p^{k-1}+p^{k-3})\circ
\sigma_{\fp,2} \\
\Desc_{\mB,p}(\Delta_{\fp}) &= 
(\gamma_{\mB,\fp,4})^{-k}\sigma_{\fp,4}.\end{split} \ee

\end{thm}

\begin{proof} This follows from Theorem \ref{desc4},
and formulas (\ref{maass1}), (\ref{for1}), and (\ref{for2}). 
Just for illustration, we include the argument in the case of $T_{\fp}$. 
Let $f \in \mM_k^{\tuM}$ and set $g=T_{\fp}f$. Fix $b \in \mB$ and write 
$b_1$ for $\sigma_{\fp,1}(b)$ and $b'$ for $b 
\a_{p}$. There exist diagonal matrices $\g \in G(\bfQ)$ and $\kappa \in K$ 
such that $b'=\g b_1 \kappa$ (hence also $b_1 = \g^{-1} b' \kappa^{-1}$). 
Identify $\mM^{\tuM}_k$ with $\prod_{b \in \mB} \mM_k^{\tuM}(G(\bfZ))$ via 
$f' \mapsto (f'_b)_{b \in \mB}$. For $\phi \in \mM_k^{\tuM}(G(\bfZ))$, $h 
\in T$, denote 
by $c_{\phi}(h)$ the $h$-Fourier coefficient of $\phi$. We will study the 
action of $T_{\fp}$ on $c_{f_{b_1}}(h)$. Since $f$ is a Maass form it is 
enough to consider $h$ of the form $\bsmat 1&* \\ *&*\esmat$. Fix such an 
$h$. Set $D=D_F \det h$. Then by (\ref{f12}) and (\ref{maass1}), 
$$c_{g_b}(h) = e^{2 \pi \tr h} c_g(h,b) = p^2 (p+1) (c_{b',f}(Dp) + 
p^{k-2} c_{b',f}(Dp^{-1})).$$ By (\ref{f12}) and (\ref{Maass condition}), 
we have \be \begin{split} c_{f_{b_1}}\left(h \bsmat 1 \\ & p \esmat 
\right) & 
= e^{2 \pi \tr 
h \bsmat 1 \\ & p \esmat} c_f \left( h \bsmat 1 \\ & p \esmat, b_1 \right) 
= 
c_{b_1,f}(Dp) = 
(\det \g^*)^k c_{b',f}(Dp)\\
 c_{f_{b_1}}\left(h \bsmat 1 \\ & 1/p \esmat \right) & = e^{2 \pi \tr
h \bsmat 1 \\ & 1/p \esmat} c_f \left( h \bsmat 1 \\ & 1/p \esmat, 
b_1 
\right) = c_{b_1,f}(Dp^{-1}) =\\
&= (\det \g^*)^k c_{b',f}(Dp^{-1}),\end{split} \ee where we have used the 
fact that $\epsilon(h) = \epsilon\left(h \bsmat 1 \\ & p \esmat\right) = 
\epsilon\left(h \bsmat 1 \\ & 1/p \esmat\right)=1$ for $h$ as above 
(the last equality holding for $h$ such that $h \bsmat 1 \\ & 1/p
\esmat \in T$). This gives us $$c_{g_b}(h) = p^2 (p+1) (\det \g^*)^{-k} 
(c_{b_1,f}(Dp) + p^{k-2} c_{b_1,f}(Dp^{-1})).$$ The claim now follows from 
Theorem \ref{desc4}.
\end{proof}

For completeness we also include the analogue of Theorem \ref{heckedesc} 
for an inert 
$p$. It can be proved in the same way or can be 
deduced 
from the results of section 3 of \cite{Gritsenko90}.

Let $p$ be an inert prime. The Hecke algebra $\mH_p$ is generated by the 
double cosets $T_{\fp}:=K_p \diag(p^{-1}, 1, p, 1 ) K_p$ and $U_{\fp}:=K_p 
 \diag (p^{-1}, p^{-1}, p,p ) K_p$. 

\begin{thm} \label{heckedesc2} Let $p$ be a rational prime which is inert  in
$F/\bfQ$. There
exists a $\bfC$-algebra map $$\Desc_{\mB,p}: \mH_p \rightarrow \bfT_p,$$ 
such
that for
every $H \in \mH_p$ the
following diagram
$$\xymatrix@C7em{\mM_k \ar[r]^{H} \ar[d]^{\Desc_{\mB}} & \mM_k
\ar[d]_{\Desc_{\mB}} \\ \prod_{b \in \mB} M_{k-1} (D_F, \chi_F)
\ar[r]^{\Desc_{\mB,p}(H)} & \prod_{b \in \mB} M_{k-1} (D_F,
\chi_F)}$$ commutes.  Moreover, one has \be \begin{split}
\Desc_{\mB,p}(T_{\fp}) & = p^{-k+4}(p^2+1) T_p^2+ p^4 + p^3+ p 
- 1,\\
\Desc_{\mB,p}(U_{\fp}) & = p^8(T_p^4 + (p+3)p^{k-2} T_p^2 + 
p^{2k-4}(p^2+p+1)).\end{split} \ee \end{thm}

\bibliography{standard1}

\end{document}